\documentclass{amsart}

\usepackage[vmargin=3cm, hmargin=3cm]{geometry}
\usepackage[foot]{amsaddr}

\usepackage{multicol}
\usepackage{afterpage}
\usepackage{makecell}
\usepackage{enumitem}
\usepackage{hyperref}
\usepackage{mathrsfs}  
\usepackage{latexsym}
\usepackage{enumitem}
\usepackage{bbm} \usepackage{yhmath}

\usepackage{amsmath, amsfonts, amssymb, amsthm, tensor}
\usepackage[all]{xy}
\usepackage{verbatim}
\usepackage{tikz-cd}
\usepackage[T1]{fontenc}
\usepackage[utf8]{inputenc}
\usepackage{multicol}
\usepackage{hhline}
\usepackage{color}
\usepackage{adjustbox}

\usepackage{mathtools}

\def\co{\colon\thinspace}

\usepackage{lineno}

\numberwithin{equation}{section}

\newtheorem{teorema}{Theorem}[section]  

\newtheorem{lem}[teorema]{Lemma}

\theoremstyle{definition}
\newtheorem{defin}[teorema]{Definition}

\theoremstyle{remark} 
\newtheorem{ex}[teorema]{Example}
\newtheorem{oss}[teorema]{Remark}

\newcommand{\C}{\mathbb{C}}
\newcommand{\R}{\mathbb{R}}
\newcommand{\N}{\mathbb{N}}
\newcommand{\Z}{\mathbb{Z}}
\newcommand{\Q}{\mathbb{Q}}
\newcommand{\T}{\mathbb{T}}

\newcommand{\trop}{\textup{trop}}

\newcommand{\B}{\mathbb{B}}

\newcommand{\num}{\mathrm{num}}

\newcommand{\pk}{K[\![t]\!]}

\newcommand{\diff}[2]{#1\{x_1, \ldots , x_{#2}\}}

\newcommand{\Sol}{\textup{Sol}}

\newcommand{\supp}{\mathrm{supp}}
\newcommand{\ord}{\mathrm{ord}}

\usepackage{todonotes}

\title{The fundamental theorem of tropical differential algebra for formal Puiseux series}

\author{Sebastian Falkensteiner}
\email{sebastian.falkensteiner@marsh.com} 
\address{Mercer Austria GmbH, Vienna, Austria}

\author{Stefano Mereta}
\email{stefano.mereta@cunef.edu}
\address{Departamento de Matemáticas, CUNEF Universidad, Calle de Almansa, 101, 28040 Madrid}

\date{\today}


\begin{document}
\begin{abstract}
The fundamental theorem of tropical differential algebra has been established for formal power series solutions of systems of algebraic differential equations. It has been shown that the direct extension to formal Puiseux series solutions fails. In this paper, we overcome this issue by transforming the given differential system and such a generalization is presented. 
Moreover, we explain why such transformations do not work for generalizing the fundamental theorem to transseries solutions, and show that one inclusion still holds for this case even without using any transformation.
\end{abstract}
\maketitle

\section*{Introduction}

Even though the first appearances of tropical methods in the area of optimization can be traced back to the 1970's, it is with the beginning of the new millennium that these techniques have started to be applied more and more often, and fruitfully, to geometry, eventually developing into an independent and rich new branch of mathematics, with connections spreading as far as computational biology and machine learning. The fundamental notion linking the classical world and the tropical one is that of tropicalization: by means of a valuation, this process moves problems of algebro-geometric nature to tropical ones, the latter having an intrinsic combinatorial and polyhedral nature. A comprehensive text to approach the subject is~\cite{macsturm}.

More recently, following the revival of interest around algebraic methods for differential equations, and differential algebra in general, see e.g.~\cite{robertz2014formal,lange2014counting,falkensteiner2023algebro}, a tropical approach to the study of solutions of ODEs has been introduced in~\cite{grig}. In this work the author gives a definition of tropical differential equations and other objects needed to state a fundamental theorem (similar to the fundamental theorem of tropical geometry, see for example~\cite[Theorem 3.2.3]{macsturm}) in this context. This theorem was successfully proven in~\cite{AGT} and subsequently extended to the case of partial differential equations in~\cite{falkfund}, see~\cite{boulier2021relationship} for a more algorithmic description. Drawing inspiration from recent works where tropical geometry is developed as geometry over the tropical semiring, such as~\cite{eqtrop, tropideal}, a more general framework encompassing that of~\cite{grig} has been proposed in~\cite{gianmereta}. This allows for a tropical theory of differential equations to be developed also in the case of non-trivially valued fields of coefficients. Within this context the authors of~\cite{merfund} prove a fundamental theorem for tropical differential algebra extending that of~\cite{AGT}. We will state our results in this language.

The fundamental theorem of tropical differential equations states equality between the (weighted) support of formal power series solutions of a differential ideal and the tropical solutions of the corresponding tropicalized differential system over a uncountable algebraically closed  valued field $K$ of characteristic 0, see Theorem~\ref{thm:ft_powerseries}. 
The direct extension to more general formal series solutions such as formal Laurent series and formal Puiseux series fails as it is presented in~\cite[Section 7]{falkfund}. 
In the present work, we use transformations of the differential ideal and its solution set to overcome this issue.

Transformations of differential equations and their solution sets are well studied in the literature; in particular, diffeomorphisms and homeomorphisms of dynamical systems have received considerable attention. 
For non-autonomous algebraic differential equations and rational solutions see~\cite[Section 4]{falkensteiner2023algebro} and references therein. 
In this paper, we use the notion of pullbacks. More precisely, we use transformations
\[
\sigma(\mathbf{s},\mathbf{z})=(\kappa(\mathbf s),\mathbf p(\mathbf s)\mathbf z)
\]
where $\kappa$ is a monomial map in the independent variables and $\mathbf p(\mathbf s)$ is a diagonal matrix of monomials multiplying on the dependent variables. 
If $G(\mathbf t,\mathbf x,\{\partial_{\mathbf t}^{\alpha}\mathbf x\}_{|\alpha|\le r})=0$ is a differential system, then its transform is obtained by pullback along $\sigma$: one substitutes $\mathbf t=\kappa(\mathbf s)$ and $\mathbf x=\mathbf p(\mathbf s)\mathbf z$ and the derivations by the chain rule, in order to obtain a new differential system in the variables $(\mathbf s,\mathbf z)$. The key point is that for suitable choices of $\kappa$ and $\mathbf p$, Puiseux series solutions of the original system correspond bijectively to formal power series solutions of the transformed one.

For applying the transformation we need the powers in $\kappa$ and $\mathbf p$ to translate the exponents to integers. 
In concrete situations they can often be detected by methods such as the Newton polygon method for differential equations, see e.g.~\cite{aroca2003power}, although no general algorithm is known for determining minimal choices.

In Section~\ref{sec-generalizationPuiseux} we first recall the classical power series version of the fundamental theorem. We then introduce multivariate formal Puiseux series and the pullback transformation that sends Puiseux solutions to formal power series solutions of a transformed system. 
This yields our generalization of the fundamental theorem to Puiseux series in both the multivariate and ordinary settings. In Section~\ref{sec-generalizationLogarithms} we explain why analogous transformations cannot be used to reduce general transseries solutions to formal power series solutions. 
Nevertheless, one inclusion in the fundamental theorem continues to hold for transseries without any transformation. We conclude with applications of the transformation method in approximation theory.

\section{Generalizing the fundamental theorem to formal Puiseux series}\label{sec-generalizationPuiseux}
The fundamental theorem of tropical (partial) differential equations is formulated and proven for (multi-variate) formal power series solutions of systems of algebraic differential equations with formal power series coefficients over an uncountable, algebraically closed field $K$, of characteristic zero. 

In this document, we generalize its scope to to formal Puiseux series solutions. The applied transformations change the differential system and the solution set so that we can apply the fundamental theorem for the case of formal power series solutions. 
Since this transformation is one-to-one, we can recover the original system with its formal Puiseux series solutions.

As we want to use the language of semirings to state our definitions and results, we make a quick recap about semirings and valuations. A \emph{semiring} $(S, \oplus, \odot)$ is an algebraic structure satisfying all the axioms to be a ring, except for the existence of additive inverses. A semiring is \emph{idempotent} if for every $a \in S$ we have that $a \oplus a = a$. 
\begin{ex}
\begin{enumerate}
\item The set $\B = \{0,\infty\}$ equipped with operations $\oplus := \min$ and $\odot := +$ is an idempotent semiring, called the idempotent semiring of boolean numbers;
\item The set $\T = \R \cup \{\infty\}$ equipped with operations $\oplus := \min$ and $\odot: = +$ is an idempotent semiring, called the idempotent semiring of tropical numbers. It contains $\B$ as a subsemiring;
\item For every $n \ge 1$ let $\T_n = \R^n \cup \{\infty\}$ with $\oplus$ given by lexicographic minimum and $\odot := +$. It is an idempotent semiring, called the idempotent semiring of rank $n$ tropical numbers.
\end{enumerate}
\end{ex}
A sum of elements in an idempotent semiring \emph{tropically vanishes} if by deleting any of the summands its result does not change.
We adopt the following generalized definition of valuation, as introduced and used in~\cite{eqtrop, gianmereta, merfund}: a \emph{valuation} on a ring $R$ is a map $v\co R \to S$ to an idempotent semiring $S$ satisfying the following conditions for all $a,b \in R$:
\begin{multicols}{2}
\begin{enumerate}
\item $v(0)  = 0_S$;
\item $v(1) = v(-1) = 1_S$;
\item $v(ab) = v(a) \odot v(b)$;
\item $v(a+b) \oplus v(a) \oplus v(b)$ tropically vanishes. 
\end{enumerate}
\end{multicols}
Let us notice that condition (4) is equivalent to saying that the minimum of the terms $v(a+b), v(a), v(b)$ is achieved at least twice.

Let $K$ be an uncountable algebraically closed field of characteristic zero equipped with a valuation $w_K \co K \to \T$ and let us denote as $\Gamma_K \subseteq \T$ its value group. For an ordered group $(\Gamma, \cdot, \le)$ we will denote as $\widetilde \Gamma$ the semiring $(\Gamma \cup \{\infty\}, \oplus, \odot)$ where $\oplus$ is the minimum with respect to $\le$ and $\odot$ is the usual operation in $\Gamma$.
In particular we denote as $\widetilde \Gamma_K$ the valuation semifield of $w_K$.

Let $R_{m,n} ^K= K[\![t_1,\ldots,t_m]\!]\{x_1,\ldots,x_n\}$ denote the differential ring of differential polynomials in $x_1,\ldots,x_n$ with formal power series coefficients. We will omit the superscript $K$ if it is clear from the context. 
Moreover, setting $\mathbf{t}:=(t_1,\ldots,t_m)$, we write $\Sol_A$ for the set of solutions in a differential algebra $A$ over $K[\![\mathbf t]\!]$. 
In the following, we will use multi-index notations. Coherently, operations such as multiplication, raising powers etc. are performed componentwise.

For simplicity of treatment, here we will introduce the objects and definitions appearing in the statement of the fundamental theorem of tropical differential equations only in the case of ODEs (a treatment of the case of PDEs can be found in~\cite{falkfund} when $w_K$ is the trivial valuation and in~\cite{merfund} in full generality).

In this case, the valuation $w_K$ on $K$ induces a valuation $v_K \co \pk \to \T_2$ defined as
\[
A_{n_0}t^{n_0} + \dots \mapsto (n_0, w_K(A_{n_0})) \in \T_2.
\]
Let us denote as $S_K$ the semiring $\widetilde{(\N \times \Gamma_K)}$, i.e.\ the image of $v_K$ in $\T_2$. When $w_K$ is a non-trivial valuation, $v_K$ is a genuinely rank 2 valuation. Otherwise $S_K$ is isomorphic to $\widetilde \N$ and we recover the rank 1 $t$-adic valuation used in~\cite{grig} and~\cite{AGT}.

Given a differential polynomial $f \in R_{1,n}$ of order less or equal to $r \in \Z_{\ge 0}$, we can write it as $f(x) = \sum_{\lambda \in \Lambda} A_\lambda x^\lambda$ where $\Lambda$ is a finite set of matrices in $\text{Mat}_{r \times n}(\Z_{\ge 0})$, $A_\lambda$ is an element of $\pk$ for every $\lambda$ and $x^\lambda$ is the differential monomial defined as $ \prod_{i,j} ( x_i^{(j)} ) ^{\lambda_{i,j}}$. 
The tropicalization $\trop_{v_K}(f)$ of $f$ with respect to $v_K$ is defined as the polynomial with coefficients in $\T_2$ obtained by applying $v_K$ coefficientwise to $f$:
\[
\trop_{v_K}(f) := \sum_{\lambda \in \Lambda} v(A_\lambda) x^\lambda= \sum_{\lambda \in \Lambda} (\alpha_{\lambda} \text{, } \beta_{\lambda}) x^\lambda
\] 
for the same $\Lambda$ and for $(\alpha_{\lambda} \text{, } \beta_{\lambda}) \in \R^2$ for every $\lambda \in \Lambda$. This polynomial lives in an algebra of polynomials over $\T_2$ in countably many variables $x_i^{(j)}$'s. Even though we could not talk rigorously about elements of this algebra as tropical differential polynomials over $\T_2$ (in fact this matter is quite subtle, see \cite[Section 3]{gianmereta}), we will nevertheless refer to them as tropical differential polynomials, as this will not give rise to any ambiguity. 
Given a differential ideal $G \subseteq R_{1,n}$ its tropicalization is the set $\trop_{v_K}(G) := \{\trop_{v_K}(f) \mid f \in G\}$.

For a power series $A \in \pk$, its tropicalization is performed by applying $w_K$ coefficientwise, obtaining an element of $\widetilde \Gamma _K[\![t]\!]$. We will denote this map as $\widetilde v _K \co \pk \rightarrow \widetilde \Gamma _K[\![t]\!] $.
Lastly let $\Phi \co \widetilde \Gamma _K[\![t]\!] \rightarrow S_K$ be defined as the homomorphism of semirings
\[
B_{n_0}t^{n_0} + \dots \mapsto (n_0, B_{n_0}) \in S_K.
\]
Unless confusion might arise, we will drop the subscripts for  $v_K$ and  $\widetilde v _K$ and just write $v$ and  $\widetilde v$.
\begin{oss}
The valuation $v_K$, the map $\widetilde v_K$ and the homomorphism fit into a commutative triangle as follows:
\begin{equation} \label{diagram:diff-enh}
\begin{tikzcd}
& \widetilde \Gamma _K[\![t]\!] \arrow[d, "\Phi"]  \\
\pk \arrow[ur, "\widetilde{v}_K"] \arrow[r, swap, "v_K"] & S_K .
\end{tikzcd}
\end{equation}
This is an instance of \emph{differential enhancement} of the valuation $w_K$ (see \cite[Section 4]{gianmereta}). The maps in the diagram above can be defined analogously for Puiseux series instead of power series, obtaining an analogous commutative diagram. In this case, though,  the fundamental theorem of tropical differential equations fails, as explained in Example~\ref{ex:fail-fund-theorem}.
\end{oss}

We endow $ \widetilde \Gamma _K[\![t]\!]$  with the tropical differential $d_{w_K}$ acting as $d_{w_K}(t^n) = w_K(n)t^{n-1}$ (see \cite[2.4]{gianmereta}). We denote the differential semiring $( \widetilde \Gamma _K[\![t]\!], d_{w_K})$ simply as $ \widetilde \Gamma _K[\![t]\!]$. Notice that in the trivially valued case $d_{w_K}$ reduces to the differential $d(t^n)=t^{n-1}$. The map $\widetilde v_K$ commutes with this differential and that on $\pk$. 

We finally introduce the notion of solution to a tropical differential equation:
\begin{defin}\label{definition:tropical-solution}
Given a differential polynomial $f \in R_{1,n}$, a vector of tropical power series $B = (B_1, \dots, B_n) \in  \widetilde \Gamma _K[\![t]\!]^n$ is a \emph{solution} to the tropicalization of $f$ if the expression
\[
\trop_v(f)(B) := \bigoplus_{\lambda \in \Lambda} v(A_\lambda) \bigodot_{i,j}  ( \Phi(d_{w_K}^j(B_i)) ) ^{\odot{\lambda_{i,j}}} \in \T_2
\]
tropically vanishes. That is, when plugging in $\Phi(d^j(B_i))$ for $x_i^{(j)}$, the minimum is achieved at least twice in $\T_2$.
\end{defin}
As here we will not deal with the general case as introduced in \cite[Section 4]{gianmereta}, but only with the case of the homomorphism $\Phi$, we will use the notation $\Sol_{ \widetilde \Gamma _K[\![t]\!]}(\trop_v(f))$ for the set of solutions to the tropicalization of the differential equation $f$ in $ \widetilde \Gamma _K[\![t]\!]^n$, as introduced in Definition \ref{definition:tropical-solution} above.

We have the following results:
\begin{teorema}[\cite{AGT,falkfund}]\label{thm:ft_powerseries}
Let $K$ be an uncountable algebraically closed field of characteristic zero, and let $w_K \co K \rightarrow \mathbb B$ be the trivial valuation. Let $G$ be a differential ideal in the differential ring $R_{m,n}$. Then 
\[ \trop_{\widetilde v}(\Sol_{K[\![\mathbf{t}]\!]}(G)) = \Sol_{\B[\![\mathbf{t}]\!]}(\trop_v(G)). \]
\end{teorema}
A generalization to the non-trivially valued case holds upon passing to a valued field extension:
\begin{teorema}[\cite{merfund}]\label{thm:ft_powerseries-valued}
Let $K$ be an uncountable algebraically closed field of characteristic zero, and let $w_K \co K \rightarrow \T$ be a nontrivial valuation.  There exists an algebraically closed valued field extension $(L,w_L)$ of $(K,w_K)$ such that for every differential ideal $G$ of the differential ring $R^K_{m,n}$ the following holds:
\[ \trop_{\widetilde v _L}(\Sol_{L[\![\mathbf{t}]\!]}(G_L)) = \Sol_{ \widetilde \Gamma _L[\![t]\!]}(\trop_{v_L}(G_L)) \]
where $G_L$ is the extension of $G$ to $R^L_{m,n}$.
\end{teorema}

\section{Formal Puiseux series and pullback transformation}\label{sec-Puiseux}
We now introduce the class of multivariate Puiseux series that will be used throughout. 
For related constructions and properties of multivariate formal Laurent and
Puiseux series fields, see~\cite{monfortekauers2013,ArocaIlardi2009}.
A {\em convex rational polyhedral cone} is a subset of $\R^m$ of the form
\[C = \{ \lambda_1v_1+\cdots+\lambda_rv_r \mid \lambda_i \in \R, \lambda_i \ge 0\}, \]
where $v_1,\ldots,v_r \in \Q^m$ are vectors. 
A cone is said to be {\em strongly convex} if it contains no nontrivial linear subspaces.
Let $C$ be a strongly convex rational polyhedral cone and let $\mathbf{d}=(d_1,\ldots,d_m) \in \Z_{>0}^m$. Set $C_\mathbf{d} = C \cap (\frac{1}{d_1} \cdot \Z,\ldots,\frac{1}{d_m} \cdot \Z)$. 
Using multi-index notation, the set of formal sums $\sum_{\mu \in C_\mathbf{d}} a_\mu t^\mu$, where $a_\mu \in K$, forms a ring denoted as $K[\![C_\mathbf{d}]\!]$.
Let $\mathbf{\omega} \in \R_{>0}^m$ be a vector with rationally independent entries and let $H_\mathbf{\omega}$ be the half-space $\{\mathbf{t} \in \R^m \mid \mathbf{\omega} \cdot \mathbf{t} \ge 0\}$. 
The ring of power series with support in a convex rational polyhedral cone in $H_\mathbf{\omega}$ and integer exponents is defined by
\[ K[\![\mathbf{t}]\!]^\mathbf{\omega}
=
\bigcup_{C\subseteq H_{\mathbf{\omega}}}
K[\![C\cap\mathbb Z^m]\!]. \]
By allowing the support in some translate of $C$ we obtain the field of {\em (multivariate) formal Laurent series}
\[K(\!(\mathbf{t})\!)^\mathbf{\omega}= \bigcup_{\mathbf{l} \in \Z^m} \mathbf{t}^{\mathbf{l}}\cdot K[\![\mathbf{t}]\!]^\mathbf{\omega}.\]
Allowing denominators $\mathbf{d} \in \Z_{>0}^m$ and shifts $\mathbf{l}$, we obtain a ring denoted by $K(\!(\mathbf{t})\!)_\mathbf{d}^\mathbf{\omega}$. Taking the union over all denominator vectors $\mathbf d$ we obtain the field of {\em (multivariate) formal Puiseux series}
\[ 
    K(\!(\mathbf{t})\!)_*^\mathbf{\omega} = \bigcup_{\mathbf{d} \in \Z_{>0}^m} K(\!(\mathbf{t})\!)_\mathbf{d}^\mathbf{\omega}.
\]
Let us note that a different choice of $\mathbf{\omega}$ leads to a different representation of the Puiseux series. For a given system of differential equations, their Puiseux series solutions can thus be translated one-to-one for different $\mathbf{\omega}$'s.

Replacing the coefficients of formal Puiseux series by coefficients in a $\widetilde \Gamma_K$ gives the semiring of
tropical Puiseux series
\[
    \widetilde \Gamma_K(\!(\mathbf t)\!)_*^{\mathbf{\omega}}
    =
    \bigcup_{\mathbf d\in\mathbb Z_{>0}^m}
    \widetilde \Gamma_K(\!(\mathbf t)\!)_{\mathbf d}^{\mathbf{\omega}} .
\]
Its elements are the tropical series with coefficients in $\Gamma$ and with bounded denominators whose supports are
contained in translates of strongly convex rational polyhedral cones contained in
$H_{\mathbf{\omega}}$.


Following~\cite[Section 8]{aroca2003power}, let $\varphi$ be a formal Puiseux series with support contained in a translate
\[
\boldsymbol\ell+C_{\mathbf d}
:=
\{\boldsymbol\ell+\mu\mid\mu\in C_{\mathbf d}\}.
\]
Let
\[
\rho_{\mathbf{d}}(\mathbf{u})
:=
(u_1^{d_1},\ldots,u_m^{d_m})
\]
be the coordinatewise ramification map. After the substitution
$\mathbf{t}=\rho_{\mathbf{d}}(\mathbf{u})$, the support of $\varphi$ is contained
in
\[
D\boldsymbol\ell+D C_{\mathbf{d}}\subseteq \mathbb Z^m,
\qquad
D:=\operatorname{diag}(d_1,\ldots,d_m).
\]
Choose a monomial map
\[
\kappa(\mathbf s) = (s_1\mathbf s^{\alpha_1},\ldots,s_m\mathbf s^{\alpha_m}),
\qquad
\alpha_i\in\mathbb Z_{\geq0}^m,
\]
obtained as a finite composition of combinatorial blow-ups, such that the
image of the cone $DC$ under the induced map on exponents lies in the first
orthant. Consequently, it sends the support set $D C_{\mathbf d}\subseteq DC\cap\mathbb Z^m$ into $\mathbb Z_{\geq0}^m$. 
Then, after replacing $\mathbf{t}$ by
\[
\mathbf{t}
=
\rho_{\mathbf{d}}(\kappa(\mathbf{s}))
=
\bigl(\kappa_1(\mathbf{s})^{d_1},\ldots,
\kappa_m(\mathbf{s})^{d_m}\bigr),
\]
the support of $\varphi$ lies in a translate of the first orthant with integer exponents. 
Here and below we use multi-index notation and write $\mathbf s^{\beta}=s_1^{\beta_1}\cdots s_m^{\beta_m}$.
Let $A_\kappa\in\operatorname{Mat}_{m\times m}(\mathbb Z)$ denote the
matrix of the induced map on exponent vectors, so that
\[
(\mathbf t^\mu)\circ\kappa=\mathbf s^{A_\kappa\mu}
\qquad
\mu\in\mathbb Z^m.
\]
For a vector of Puiseux series
$\varphi=(\varphi_1,\ldots,\varphi_n)$, assume that
\[
\supp(\varphi_i)\subseteq \boldsymbol\ell_i+C_{\mathbf d}
\qquad i=1,\ldots,n.
\]
After the substitution
$\mathbf t=\rho_{\mathbf d}(\kappa(\mathbf s))$, the translation
vector of the $i$-th component becomes
\[
A_\kappa D\boldsymbol\ell_i.
\]
Accordingly, set
\[
\mathbf p(\mathbf s)
:=
\operatorname{diag}
\left(
\mathbf s^{A_\kappa D\boldsymbol\ell_1},
\ldots,
\mathbf s^{A_\kappa D\boldsymbol\ell_n}
\right).
\]
We then consider the transformation
\[
\kappa_{\mathbf d}(\mathbf s)
:=
\rho_{\mathbf d}(\kappa(\mathbf s))
=
\bigl(\kappa_1(\mathbf s)^{d_1},\ldots,
\kappa_m(\mathbf s)^{d_m}\bigr)
\]
and
\[
\sigma(\mathbf s,\mathbf z)
=
(\mathbf t,\mathbf x)
=
\bigl(\kappa_{\mathbf d}(\mathbf s),
\mathbf p(\mathbf s)\mathbf z\bigr).
\]
Thus the old variables are $(\mathbf t,\mathbf x)$ and the new ones are $(\mathbf s,\mathbf z)$, and the relation between dependent variables is
\[
(\mathbf x\circ \kappa_{\mathbf d})(\mathbf s)=\mathbf p(\mathbf s)\mathbf z(\mathbf s).
\]
Let $\mathcal J_{\mathbf s}(\kappa_{\mathbf d})=\bigl(\tfrac{\partial (\kappa_{\mathbf d})_i}{\partial s_j}\bigr)_{i,j}$ be the Jacobian of $\kappa_{\mathbf d}$. Since $\kappa$ is a monomial map obtained from combinatorial blow-ups and
$\rho_{\mathbf d}$ is a coordinatewise ramification map, $\kappa_{\mathbf d}=\rho_{\mathbf d}\circ\kappa$ is a monomial map.
It is invertible as a monomial substitution after allowing rational exponents, and
$\mathcal J_{\mathbf s}(\kappa_{\mathbf d})^{-1}$ has Laurent monomial entries with coefficients in $\mathbb Q$. 
The pullback acts on derivations by the chain rule:
\[
\sigma^*(\partial_{t_i})=\sum_{j=1}^m (\mathcal J_{\mathbf s}(\kappa_{\mathbf d})^{-1})_{j,i}\,\partial_{s_j}.
\]
Hence, for the dependent variables,
\begin{equation*}
\sigma^*(\partial_{t_i}\mathbf x)
=
\sum_{j=1}^m (\mathcal J_{\mathbf s}(\kappa_{\mathbf d})^{-1})_{j,i}
\,\partial_{s_j}(\mathbf p(\mathbf s)\mathbf z)
=
\sum_{j=1}^m (\mathcal J_{\mathbf s}(\kappa_{\mathbf d})^{-1})_{j,i}
\bigl((\partial_{s_j}\mathbf p(\mathbf s))\mathbf z+\mathbf p(\mathbf s)\partial_{s_j}\mathbf z\bigr),
\end{equation*}
and higher derivatives are obtained by iterating these operators. If $F \in R_{m,n}$ is a differential polynomial of order at most $r$, $F(t,x,\{\partial_\alpha x\}_{|\alpha| \le r})=0$ is transformed by the pullback
\[
F^\sigma(\mathbf{s},\mathbf{z},\{\partial^\alpha_\mathbf{s} \mathbf{z}\}_{|\alpha| \le r}) := \num(\sigma^*(F)(\mathbf{s},\mathbf{z})) = \num(F(\kappa_{\mathbf d}(\mathbf{s}), \mathbf{p}(\mathbf{s}) \mathbf{z}, \{\sigma^*(\partial_\mathbf{t}^\alpha \mathbf{x})\}_{|\alpha| \le r})) = 0,
\]
where $\num(\,\cdot\,)$ means the numerator considered as fractional expression in the $s_i$. Since the transformation is invertible at the formal level, a tuple $\varphi\in K(\!(\mathbf t)\!)_*^{\mathbf{\omega}}$ is a solution of $F$ if and only if the transformed tuple $\psi$, defined by
\[
\psi(\mathbf s)=\mathbf p(\mathbf s)^{-1}(\varphi\circ \kappa_{\mathbf d})(\mathbf s),
\]
is a formal power series solution of $F^{\sigma}$.

For a finite differential system $F_1,\ldots,F_N$, we denote by
\[
G^\sigma:=[F_1^\sigma,\ldots,F_N^\sigma]
\]
the differential ideal generated by its transformed equations. 
The same transformation as above can be defined analogously on boolean or tropical Puiseux or power series by acting on their supports. Since both $\kappa_{\mathbf d}$ and multiplication by monomials commute with the map $\widetilde v$, the tropicalization of a transformed solution is obtained by applying the same map on the tropicalization of the original solution. Therefore the tropicalized solution set of $G$ over multivariate Puiseux series is recovered from the tropicalized solution sets of the transformed systems $G^{\sigma}$ over formal power series.

For a fixed choice of $\mathbf{\omega}$, a given Puiseux series solution of $G$ has support in a translate of some strongly convex rational polyhedral cone contained in $H_{\mathbf{\omega}}$. The associated denominator vector, translation data, and finite sequence of combinatorial blow-ups generally depend on the chosen solution. Accordingly, in the statement below the union is taken over all pullback transformations arising from all such choices of the support cone, denominator vector, translation data, and
finite sequences of combinatorial blow-ups.

\begin{oss}\label{oss:radical-ideals}
For a differential ideal $G$, we write $\sqrt{G}$ for its radical. 
Since the ground field has characteristic zero, $\sqrt{G}$ is again a differential ideal.
In every differential field one has
\[
\Sol(G)=\Sol(\sqrt{G}).
\]
Moreover, the tropical solution sets satisfy
\[
\Sol(\trop(G))
=
\Sol(\trop(\sqrt{G})).
\]
Indeed, if $f\in\sqrt{G}$, then $f^N\in G$ for some $N \in \mathbb N$; a unique minimal term in the evaluation of $\trop(f)$ would give a unique minimal term in the evaluation of $\trop(f^N)$, contradicting the fact that the point is a tropical solution of $\trop(G)$. 
In the following, we may therefore work with radical differential ideals. 
By the Ritt--Raudenbush basis theorem~\cite{ritt1950differential,
kolchin1973differential,boulier2019ritt}, a radical differential ideal in the present setting is the radical differential closure of a finitely generated differential ideal.
\end{oss}

\begin{oss}
    Notice that given a differential ideal $G$ of $R^K_{m,n}$ and a field extension $L|K$ the transformation commutes with the extension, i.e.
    \[
        (G_L)^\sigma = (G^\sigma)_L.
    \]
    Thus we will write $G^\sigma_L$ without ambiguity. 
\end{oss}

Using the above transformation, both for trivial and nontrivial valuation, the aforecited fundamental theorem for formal power series extends to the Puiseux setting as follows. 
\begin{teorema}\label{thm:ft_PuiseuxseriesPartial}
Let $K$ be an uncountable, algebraically closed field of characteristic zero, equipped with the trivial valuation $K \rightarrow \B$. Let $G \subset R_{m,n}$ be a radical differential ideal. 
Fix a vector $\mathbf{\omega}\in\mathbb R_{>0}^m$ with rationally independent entries over $\mathbb Q$. Then 
\[
\trop_{\widetilde v}(\Sol_{K(\!(\mathbf t)\!)_*^{\mathbf{\omega}}}(G))
=
\bigcup_\sigma
\left (
\sigma^{-1}
\left(
\Sol_{\B[\![\mathbf s]\!]}(\trop_{v}(G^\sigma))
\right)
\cap
\bigl(\B(\!(\mathbf t)\!)_*^{\mathbf{\omega}}\bigr)^n
\right ),
\]
where the union ranges over all pullback transformations
\[
\sigma(\mathbf s,\mathbf z)=(\mathbf t,\mathbf x)=(\kappa_{\mathbf d}(\mathbf s),\mathbf p(\mathbf s)\mathbf z)
\]
arising from the following data: a denominator vector
$\mathbf d\in\Z_{>0}^m$, translation vectors
$\boldsymbol\ell_1,\ldots,\boldsymbol\ell_n$, a strongly convex rational polyhedral cone $C\subseteq H_{\mathbf{\omega}}$, 
and a finite composition of combinatorial blow-ups producing a monomial map $\kappa$, together with the coordinatewise ramification map
$\rho_{\mathbf d}$, where
\[
\kappa_{\mathbf d}:=\rho_{\mathbf d}\circ\kappa.
\]
\end{teorema}
\begin{teorema}\label{thm:ft_PuiseuxseriesPartialValued}
Let $K$ be an uncountable algebraically closed field of characteristic zero, equipped with a nontrivial valuation $w_K \co K \rightarrow \T$. Fix a vector $\mathbf{\omega}\in\mathbb R_{>0}^m$ with rationally independent entries over $\mathbb Q$. There exists an algebraically closed valued field extension $(L,w_L)$ of $(K,w_K)$ such that for every radical differential ideal $G \subset R^K_{m,n}$ the following holds
\[
\trop_{\widetilde v _L}(\Sol_{L(\!(\mathbf t)\!)_*^{\mathbf{\omega}}}(G_L))
=
\bigcup_\sigma
\left (
\sigma^{-1}
\left(
\Sol_{\widetilde \Gamma_L[\![\mathbf s]\!]}(\trop_{v_L}(G^\sigma_L))
\right)
\cap
\bigl(\widetilde \Gamma_L(\!(\mathbf t)\!)_*^{\mathbf{\omega}}\bigr)^n
\right ),
\]
where the union ranges over all pullback transformations $\sigma$ as above.
\end{teorema}

\begin{oss}\label{oss:sigma-on-supports}
Let
\[
\sigma(\mathbf{s},\mathbf{z})=(\mathbf{t},\mathbf{x})=(\kappa_{\mathbf d}(\mathbf{s}),\mathbf p(\mathbf{s})\mathbf{z})
\]
be a pullback transformation as in the statement of the theorem above. Besides acting on classical series, the map $\sigma$ also acts on tropical series by acting on their supports: if $A\subseteq \Q^m$ is the support of a tropical series in the variables $\mathbf{s}$, then $\sigma^{-1}(A)$ denotes the support obtained after applying the inverse monomial change of variables $\kappa_{\mathbf d}^{-1}$ and the inverse monomial factor $\mathbf p(\mathbf{s})^{-1}$. With this convention, the expressions
\[
\sigma^{-1}\Bigl(\Sol_{\B[\![\mathbf{s}]\!]}(\trop(G^{\sigma}))\Bigr) \quad \text{and} \quad \sigma^{-1}
\left(
\Sol_{\widetilde \Gamma_L[\![\mathbf s]\!]}(\trop(G^\sigma_L))
\right)
\]
above are understood as the image of the set of tropical solutions under this induced action on supports.
\end{oss}

\section{Proof of the fundamental theorem for Puiseux series}

\begin{lem}\label{lem:pullback-solutions}
Let $G=[F_1,\ldots,F_N]\subset R_{m,n}$ be a differential ideal, and let
\[
\sigma(\mathbf{s},\mathbf{z})=(\mathbf{t},\mathbf{x})=(\kappa_{\mathbf d}(\mathbf{s}),\mathbf p(\mathbf{s})\mathbf{z})
\]
be a pullback transformation as above. For $i=1,\ldots,N$, let $F_i^{\sigma}$ be obtained from $F_i$ by pullback and clearing denominators, and set
\[
G^{\sigma}:=[F_1^{\sigma},\ldots,F_N^{\sigma}].
\]
Then, for every $\varphi\in (K(\!(\mathbf{t})\!)_*^{\mathbf{\omega}})^n$ such that
\[
\psi(\mathbf{s}):=\mathbf p(\mathbf{s})^{-1}(\varphi\circ\kappa_{\mathbf d})(\mathbf{s})
\in K[\![\mathbf{s}]\!]^n,
\]
one has
\[
\varphi\in \Sol_{K(\!(\mathbf{t})\!)_*^{\mathbf{\omega}}}(G)
\qquad\Longleftrightarrow\qquad
\psi\in \Sol_{K[\![\mathbf{s}]\!]}(G^{\sigma}).
\]
\end{lem}
\begin{proof}
Fix $i\in\{1,\ldots,N\}$. By construction, $F_i^{\sigma}$ is obtained from $\sigma^*(F_i)$ by multiplication with a monomial in the variables $s_1,\ldots,s_m$. Since a monomial is not a zero divisor in $K[\![\mathbf{s}]\!]$, the equality $F_i^{\sigma}(\psi)=0$ is equivalent to $\sigma^*(F_i)(\psi)=0$.

On the other hand, the definition of pullback gives
\[
\sigma^*(F_i)(\psi)=F_i(\varphi)\circ\kappa_{\mathbf d}.
\]
As $\kappa_{\mathbf d}$ is a monomial substitution, the series $F_i(\varphi)$ vanishes if and only if $F_i(\varphi)\circ\kappa_{\mathbf d}$ vanishes. Therefore
\[
F_i(\varphi)=0
\qquad\Longleftrightarrow\qquad
F_i^{\sigma}(\psi)=0.
\]
Since this holds for every $i=1,\ldots,N$, we obtain
\[
F_i(\varphi)=0 \text{ for all } i=1,\ldots,N
\qquad\Longleftrightarrow\qquad
F_i^{\sigma}(\psi)=0 \text{ for all } i=1,\ldots,N.
\]
Now $G=[F_1,\ldots,F_N]$ and $G^{\sigma}=[F_1^{\sigma},\ldots,F_N^{\sigma}]$ are defined using these chosen generating families, so vanishing on the generators is equivalent to belonging to the corresponding solution sets.
\end{proof}

\begin{lem}\label{lem:pullback-tropicalization}
With notations as above, let $\varphi\in (K(\!(\mathbf{t})\!)_*^{\mathbf{\omega}})^n$ and assume that
\[
\psi(\mathbf{s})=\mathbf p(\mathbf{s})^{-1}(\varphi\circ\kappa_{\mathbf d})(\mathbf{s})
\in K[\![\mathbf{s}]\!]^n.
\]
Then
\[
\trop_{\widetilde v}(\varphi)=\sigma^{-1}(\trop_{\widetilde v}(\psi)).
\]
\end{lem}
\begin{proof}
The map $\kappa_{\mathbf d}$ is a monomial change of variables and $\mathbf p(\mathbf{s})$ is a diagonal matrix of monomials. Hence both operations just act on a series by translating and linearly transforming the exponents of its monomials. Consequently, the map $\widetilde v$ commutes with both operations.

As $\widetilde v$ is defined as the coefficientwise application of $w_K$, applying this componentwise to the relation
\[
\psi(\mathbf{s})=\mathbf p(\mathbf{s})^{-1}(\varphi\circ\kappa_{\mathbf d})(\mathbf{s})
\]
shows that
\[
\trop_{\widetilde v}(\psi)=\sigma(\trop_{\widetilde v}(\varphi)).
\]
Applying the inverse support transformation gives $\trop_{\widetilde v}(\varphi)=\sigma^{-1}(\trop_{\widetilde v}(\psi))$.
\end{proof}

\begin{oss}\label{lem:finite-generators-under-sigma}
Let $G=[F_1,\ldots,F_N]\subset R_{m,n}$ be a differential ideal, and let
\[
\sigma(\mathbf{s},\mathbf{z})=(\mathbf{t},\mathbf{x})=(\kappa_{\mathbf d}(\mathbf{s}),\mathbf p(\mathbf{s})\mathbf{z})
\]
be a pullback transformation as above. By the definition of the transformed differential ideal
\[
G^{\sigma}:=[F_1^{\sigma},\ldots,F_N^{\sigma}],
\]
the finite generating families are preserved under the transformation.
\end{oss}

\begin{proof}[Proof of Theorem~\ref{thm:ft_PuiseuxseriesPartial} and Theorem~\ref{thm:ft_PuiseuxseriesPartialValued}]
By the Ritt--Raudenbush basis theorem~\cite{ritt1950differential,
kolchin1973differential,boulier2019ritt}, there exist differential
polynomials $F_1,\ldots,F_N$ such that
\[
G=\sqrt{[F_1,\ldots,F_N]}.
\]
By Remark~\ref{oss:radical-ideals}, the (classical and tropical) solution sets of $G$ coincide with those of  $[F_1,\ldots,F_N]$. 
Let $(L,w_L)$ be the valued field extension of $(K,w_K)$ as in Theorem \ref{thm:ft_powerseries-valued}. When $w_K$ is the trivial valuation we set $L=K$. We apply the pullback construction to the finite system
\[
G_L^\sigma:=[F_1^\sigma,\ldots,F_N^\sigma] \subseteq R_{m,n}^L.
\]

We now prove the two inclusions separately. 
For ``$\subseteq$'', let
\[
\varphi\in \Sol_{L(\!(\mathbf{t})\!)_*^{\mathbf{\omega}}}(G_L).
\]
By definition of $L(\!(\mathbf{t})\!)_*^{\mathbf{\omega}}$, the support of each component of $\varphi$ is contained in a translate of a strongly convex rational polyhedral cone contained in $H_{\mathbf{\omega}}$, with some denominator vector. Choosing the corresponding monomial map $\kappa_{\mathbf d}$ and monomial factor $\mathbf p$ as in the construction above, we obtain a pullback transformation
\[
\sigma(\mathbf{s},\mathbf{z})=(\mathbf{t},\mathbf{x})=(\kappa_{\mathbf d}(\mathbf{s}),\mathbf p(\mathbf{s})\mathbf{z})
\]
such that
\[
\psi(\mathbf{s}):=\mathbf p(\mathbf{s})^{-1}(\varphi\circ\kappa_{\mathbf d})(\mathbf{s})
\in L[\![\mathbf{s}]\!]^n.
\]
By Lemma~\ref{lem:pullback-solutions},
\[
\psi\in \Sol_{K[\![\mathbf{s}]\!]}(G^{\sigma}).
\]
Applying Theorem~\ref{thm:ft_powerseries-valued} (respectively Theorem~\ref{thm:ft_powerseries} in the trivially valued case) to the transformed system yields
\[
\trop_{\widetilde v _L}(\psi)\in \Sol_{\widetilde \Gamma_L[\![\mathbf{s}]\!]}(\trop_{v_L}(G^{\sigma}_L)).
\]
By Lemma~\ref{lem:pullback-tropicalization},
\[
\trop_{\widetilde v _L}(\varphi)=\sigma^{-1}(\trop_{\widetilde v _L}(\psi)),
\]
so $\trop_{\widetilde v _L}(\varphi)$ belongs to the right-hand side.

For the other inclusion, let
\[
A\in
\sigma^{-1}\Bigl(\Sol_{\widetilde \Gamma_L[\![\mathbf{s}]\!]}(\trop_{v_L}(G^{\sigma}_L))\Bigr)
\cap
\bigl(\widetilde \Gamma_L(\!(\mathbf t)\!)_*^{\mathbf{\omega}}\bigr)^n
\]
for some pullback transformation $\sigma$ as in the statements above
(where $L=K$ and $\widetilde \Gamma_L = \widetilde \Gamma_K = \B$ in the trivially valued case). Then
\[
\sigma(A)\in \Sol_{\widetilde \Gamma _L[\![\mathbf{s}]\!]}(\trop_{v_L}(G^{\sigma}_L)).
\]
By Theorem~\ref{thm:ft_powerseries}, there exists
\[
\psi\in \Sol_{L[\![\mathbf{s}]\!]}(G^{\sigma}_L)
\]
such that
\[
\trop_{\widetilde v _L}(\psi)=\sigma(A).
\]
Define now
\[
\varphi:=\bigl(\mathbf p(\mathbf{s})\psi(\mathbf{s})\bigr)\circ\kappa_{\mathbf d}^{-1}.
\]
By the defining intersection on the right-hand side,
\[
A\in\bigl(\widetilde \Gamma_K(\!(\mathbf t)\!)_*^{\mathbf{\omega}}\bigr)^n.
\]
Moreover,
\[
\trop_{\widetilde v _L}(\varphi)
=
\sigma^{-1}(\trop_{\widetilde v _L}(\psi))
=
A.
\]
Hence the support of every component of $\varphi$ is contained in a translate of a strongly convex rational polyhedral cone in $H_{\mathbf{\omega}}$, with bounded denominators. Since all exponents are rational,
\[
\varphi\in
\bigl(L(\!(\mathbf t)\!)_*^{\mathbf{\omega}}\bigr)^n.
\]
Moreover, by construction,
\[
\psi(\mathbf{s})=\mathbf p(\mathbf{s})^{-1}(\varphi\circ\kappa_{\mathbf d})(\mathbf{s}),
\]
so Lemma~\ref{lem:pullback-solutions} gives
\[
\varphi\in \Sol_{L(\!(\mathbf{t})\!)_*^{\mathbf{\omega}}}(G).
\]
Finally, Lemma~\ref{lem:pullback-tropicalization} yields
\[
\trop_{\widetilde v _L}(\varphi)=\sigma^{-1}(\trop_{\widetilde v _L}(\psi))
=\sigma^{-1}(\sigma(A))=A.
\]
\end{proof}

\begin{ex}
Let
\[
F_1=t_1\frac{\partial x}{\partial t_1}
+t_2\frac{\partial x}{\partial t_2}
-\frac{x}{2},
\qquad
F_2=\frac{\partial x}{\partial t_1}
-\frac{\partial x}{\partial t_2}
\]
be differential polynomials in $\C[\![t_1,t_2]\!]\{x\}$.
We first note that the system has no nonzero formal power series
solutions. Indeed, from
\[
\frac{\partial^{i+j}}{\partial t_1^i\partial t_2^j}F_1
=
t_1\frac{\partial^{i+j+1}x}{\partial t_1^{i+1}\partial t_2^j}
+t_2\frac{\partial^{i+j+1}x}{\partial t_1^i\partial t_2^{j+1}}
+\Bigl(i+j-\frac12\Bigr)
\frac{\partial^{i+j}x}{\partial t_1^i\partial t_2^j},
\]
one directly sees that $x=0$ is the only possible formal power series
solution.

Nevertheless, the system admits Puiseux series solutions. Consider the
solution
\[
x(t_1,t_2)=c\sqrt{t_1+t_2}
=
c\sum_{i\geq0}\binom{1/2}{i}t_1^i t_2^{1/2-i}.
\]
Its support is contained in a translate of a rational polyhedral cone with
denominator vector
\[
\mathbf d=(1,2).
\]
Let
\[
\rho_{\mathbf d}(u_1,u_2)=(u_1,u_2^2)
\]
and choose the monomial map
\[
\kappa(s_1,s_2)=(s_1^2s_2,s_1).
\]
The corresponding monomial substitution is
\[
\kappa_{\mathbf d}
=
\rho_{\mathbf d}\circ\kappa,
\qquad
\kappa_{\mathbf d}(s_1,s_2)
=
(s_1^2s_2,s_1^2).
\]
We use the monomial factor
\[
p(s_1,s_2)=s_1
\]
and consider the pullback transformation
\[
\sigma(s_1,s_2,z)
=
(t_1,t_2,x)
=
(s_1^2s_2,s_1^2,s_1z).
\]
Thus
\[
(x\circ\kappa_{\mathbf d})(s_1,s_2)=s_1z(s_1,s_2).
\]
Since
\[
\mathcal J_{\mathbf s}(\kappa_{\mathbf d})
=
\begin{pmatrix}
2s_1s_2 & s_1^2\\
2s_1 & 0
\end{pmatrix},
\qquad
\mathcal J_{\mathbf s}(\kappa_{\mathbf d})^{-1}
=
\begin{pmatrix}
0 & \frac{1}{2s_1}\\
s_1^{-2} & -s_2s_1^{-2}
\end{pmatrix},
\]
the pulled-back derivations are
\[
\sigma^*(\partial_{t_1})
=
s_1^{-2}\partial_{s_2},
\qquad
\sigma^*(\partial_{t_2})
=
\frac{1}{2s_1}\partial_{s_1}
-s_2s_1^{-2}\partial_{s_2}.
\]
Applying them to $x=s_1z$, we obtain
\[
\sigma^*\Bigl(\frac{\partial x}{\partial t_1}\Bigr)
=
s_1^{-1}\frac{\partial z}{\partial s_2},
\]
and
\[
\sigma^*\Bigl(\frac{\partial x}{\partial t_2}\Bigr)
=
\frac{1}{2s_1}z
+\frac12\frac{\partial z}{\partial s_1}
-s_2s_1^{-1}\frac{\partial z}{\partial s_2}.
\]
Substituting these expressions into $F_1$ and $F_2$, and clearing
nonzero monomial factors, yields the transformed system
\[
F_1^\sigma
=
\frac{\partial z}{\partial s_1},
\qquad
F_2^\sigma
=
2(1+s_2)\frac{\partial z}{\partial s_2}
-s_1\frac{\partial z}{\partial s_1}
-z.
\]
Hence every formal power series solution of the transformed system is
independent of $s_1$. In particular,
\[
z(s_1,s_2)
=
c\sum_{i\geq0}\binom{1/2}{i}s_2^i
=
c(1+s_2)^{1/2}
\]
is a formal power series solution of
$[F_1^\sigma,F_2^\sigma]$.

Passing back to the original variables gives
\[
x(t_1,t_2)
=
s_1z(s_1,s_2)
\big|_{s_1=t_2^{1/2},\,s_2=t_1/t_2}
=
c\sqrt{t_1+t_2}.
\]
At the tropical level, the support
\[
\supp(z)=\{(0,i)\mid i\in\mathbb N\}
\]
is sent by $\sigma^{-1}$ to
\[
\trop_{\widetilde v}(x)
=
\left\{
\left(i,\frac12-i\right)
\;\middle|\;
i\in\mathbb N
\right\},
\]
which is the support of the Puiseux series
$c\sqrt{t_1+t_2}$ and is a tropical solution of
$\trop([F_1,F_2])$.
\end{ex}

\begin{oss}\label{oss:necessity-support-intersection}
The intersection with $\bigl(\B(\!(\mathbf t)\!)_*^{\mathbf{\omega}}\bigr)^n$ on the right-hand side of Theorem~\ref{thm:ft_PuiseuxseriesPartial} is
necessary and natural. 
A pullback transformation is chosen so that the support cone of a prescribed Puiseux series is sent into the first orthant. 
This does not imply that the inverse image of the whole first orthant is contained in $H_{\mathbf{\omega}}$.
Consequently, a boolean formal power series solution of the transformed system can pull back to a boolean Puiseux series whose support does not belong to the fixed $\mathbf{\omega}$-Puiseux class.

For example, take $m=2$, let
\[
\mathbf{\omega}=(\sqrt{2},1),
\qquad
\kappa(s_1,s_2)=(s_1,s_1s_2),
\]
and take trivial ramification and no monomial factor in the dependent variable. 
The induced map on exponent vectors is
\[
A_\kappa=
\begin{pmatrix}
1&1\\
0&1
\end{pmatrix}.
\]
The cone
\[
C=\mathbb R_{\geq0}(1,0)
\]
is contained in $H_{\mathbf{\omega}}$ and is mapped into the first orthant.
However, the boolean formal power series
\[
B=\sum_{k\geq0}s_2^k
\]
has support $\mathbb N(0,1)$, and its inverse image under $\kappa$ has
support
\[
\mathbb N(-1,1).
\]
Since
\[
\mathbf{\omega}\cdot(-1,1)=1-\sqrt{2}<0,
\]
this support is not contained in a translate of a cone in
$H_{\mathbf{\omega}}$. Hence
\[
\sigma^{-1}(B)\notin
\B(\!(t_1,t_2)\!)_*^{\mathbf{\omega}}.
\]
For the zero differential ideal, $B$ is a tropical solution of the transformed system. 
Thus, without the intersection in the theorem, the right-hand side would contain tropical series that cannot occur on the left-hand side.
\end{oss}

\begin{oss}
The theorem involves a union over support cones, denominator vectors, translation data, and combinatorial blow-up charts.
In classical algebraic geometry, Gröbner fans provide finite polyhedral decompositions associated with a fixed polynomial ideal. 
Determining whether, for suitable classes of differential systems, the pullback transformations needed in Theorem~\ref{thm:ft_PuiseuxseriesPartial} can be controlled by a finite polyhedral collection of cones or charts is an open question.
\end{oss}

\subsection{Ordinary case as a specialization}
The usual univariate Puiseux transformations are recovered from the multivariate construction by taking $m=1$. In the ordinary case, the combinatorial part of the construction disappears: every monomial map obtained from combinatorial blow-ups is simply a power map
\[
\kappa_d(s)=s^d
\]
for some $d\in\Z_{>0}$, and the diagonal factor reduces to the monomial
\[
p(s)=s^{dl}.
\]
Moreover, the choice of the weight vector $\mathbf{\omega}$ is immaterial in one variable, up to sign. 
Indeed, the relevant half-space is determined by the sign of the unique coordinate, so one may take $w=1$; the opposite choice $w=-1$ corresponds to replacing the parameter by its inverse.

Then the pullback map is
\[
\sigma_{d,l}(s,z)=(t,x)=(s^d,s^{dl}z),
\]
and therefore
\[
z=s^{-dl}(x\circ\kappa_d)=s^{-dl}x(s^d).
\]
In particular, if $\ord_t(x)\geq l$, then
\[
\ord_s(z)=d\,\ord_t(x)-d\,l\geq0,
\]
so $z$ is a formal power series. This is the standard ramified substitution used for Puiseux series.

\begin{ex}
Consider the differential equation
\[
F=xx'-1
\]
in $K[\![t]\!]\{x\}$, together with the initial condition $x(0)=0$.
Its solutions satisfying this initial condition are
\[
x(t)=\pm\sqrt{2t}=\pm\sqrt{2}\,t^{1/2}.
\]
They are Puiseux series of order $1/2$, hence have order bounded below by
$0$.

We use the ordinary pullback transformation with $d=2$ and $l=0$:
\[
\sigma_{2,0}(s,z)=(t,x)=(s^2,z).
\]
Thus
\[
z=x\circ\kappa_2=x(s^2).
\]
Since
\[
\sigma_{2,0}^*(\partial_t)=\frac{1}{2s}\partial_s,
\]
we obtain
\[
\sigma_{2,0}^*(x')
=
\frac{1}{2s}z'.
\]
Consequently,
\[
\sigma_{2,0}^*(F)
=
\frac{1}{2s}zz'-1.
\]
Clearing denominators yields the transformed differential equation
\[
F^{\sigma_{2,0}}=zz'-2s.
\]

Now
\[
z(s)=\pm\sqrt{2}\,s
\]
is a formal power series solution of
\[
F^{\sigma_{2,0}}=0.
\]
Pulling it back gives
\[
x(t)=z(t^{1/2})
=\pm\sqrt{2}\,t^{1/2}.
\]
At the tropical level, the support $\{1\}$ of $z$ is sent by
$\sigma_{2,0}^{-1}$ to the support $\{1/2\}$ of $x$.
\end{ex}

\section{Can we generalize the fundamental theorem to series involving logarithms?}\label{sec-generalizationLogarithms}
We now consider series involving logarithms. For this purpose we use grid-based transseries, which generalize formal Puiseux series over $\R$ and, after complexification, over $\C$. Informally, transseries are formal Hahn series in real powers of the indeterminate $t$, logarithms, exponentials, and their finite compositions, with well-based support. For rigorous definitions we refer to~\cite{edgar2009transseries,van2006transseries}; the complex case is treated in~\cite{van2001complex}.

Let $(\mathcal M,\cdot)$ be a totally ordered monomial group over a real trigonometric field $K$ such as $\R$, or over its complexification. We think of $\mathcal M$ as generated by monomials in the variables $\mathbf{t}$ together with logarithms and exponentials in $\mathbf{t}$. A grid-based transseries is a series
\[
\varphi=\sum_{m\in\mathcal M}\varphi_m\,m
\]
whose support
\[
\supp(\varphi)=\{m\in\mathcal M\mid \varphi_m\neq 0\}
\]
is finitely generated in $\mathcal M$. We denote the corresponding field by $K[\![\mathcal M]\!]$.

Let $d\co \mathcal M\to K[\![\mathcal M]\!]$ be a map satisfying the Leibniz rule and such that $d(\log(m))=d(m)/m$ for all $m\in\mathcal M$. By~\cite[Theorem 5.1]{van2006transseries}, the map $d$ extends to an exp-log derivation on $K[\![\mathcal M]\!]$, thus making $(K[\![\mathcal M]\!],d)$ into a differential field.

Let $\leq$ denote the total order on $\mathcal M$. Then $\mathcal M\cup\{\infty\}$ becomes an idempotent semifield by setting
\[
m\oplus n:=\min_{\leq}\{m,n\},
\qquad
m\odot n:=mn,
\]
and extending these operations by $m\oplus\infty=m$ and $m\odot\infty=\infty$. We equip $K[\![\mathcal M]\!]$ with the assignment
\[
v\co K[\![\mathcal M]\!]\to \mathcal M\cup\{\infty\},
\qquad
v(\varphi):=\min_{\leq}\supp(\varphi).
\]
It is straightforward to check that $v$ is a valuation.

Consider now the idempotent semiring $\B[\![\mathcal M]\!]$ of boolean transseries. Endowed with the differential induced by $d$, the map
\[
\Phi\co (\B[\![\mathcal M]\!],d)\to \mathcal M\cup\{\infty\},
\qquad
\Phi(A)=\min_{\leq}\supp(A),
\]
defines a tropical pair $\mathbf S$ in the sense of~\cite{gianmereta}. This pair is not reduced; indeed, its restriction to $\B(\!(t)\!)$ is the usual non-reduced tropical pair $\B(\!(t)\!)\to \Z\cup\{\infty\}$.

We also define
\[
\widetilde v\co K[\![\mathcal M]\!]\to \B[\![\mathcal M]\!]
\]
by coefficientwise application of the trivial valuation on $K$, that is, by sending a transseries to its support. By construction, $\widetilde v$ commutes with the differentials. Hence we obtain a commutative diagram
\[
\begin{tikzcd}
& \B[\![\mathcal M]\!] \arrow[d, "\Phi"] \\
K[\![\mathcal M]\!] \arrow[ur, "\widetilde v"] \arrow[r, swap, "v"] & \mathcal M\cup\{\infty\}.
\end{tikzcd}
\]
This satisfies all axioms of a differential enhancement of the valuation $v$, except for the reducedness of $\Phi$. Denote this enhancement by
\[
\mathbf v=(v,\widetilde v)\co (K[\![\mathcal M]\!],d)\to \mathbf S.
\]
By the same argument as in~\cite[Proposition 5.2.2]{gianmereta}, reducedness is not needed for the inclusion
\[
\trop_{\widetilde v}(\Sol_{K[\![\mathcal M]\!]}(I))
\subseteq
\Sol_{\mathbf S}(\trop_v(I))
\]
for every differential ideal $I\subset \diff{K[\![\mathcal M]\!]}{n}$.

We now show that, in general, equality fails.
\begin{ex}\label{ex:fail-fund-theorem}
Consider the differential polynomial
\[
F=tx'-x-t\in \C[\![t]\!]\{x\}.
\]
The differential ideal $[F]$ is prime, and the solutions of $F=0$ are
\[
\varphi=ct+t\log t,
\qquad c\in\C.
\]
Therefore
\[
\trop_{\widetilde v}(\Sol_{\C[\![\mathcal M]\!]}([F]))
=
\{t\log t,\, t+t\log t\}.
\]
Here we regard $\C[\![t,\log(t)]\!]$ as a differential subring of $\C[\![\mathcal M]\!]$, where $\mathcal M$ is the ordered monomial group generated by $t$ and $\log(t)$, with order chosen so that $t<\log(t)$. The derivation satisfies $d(\log t)=t^{-1}$.

As an ordered idempotent semiring, $\mathcal M\cup\{\infty\}$ is isomorphic to $\Z^2\cup\{\infty\}\subset \T_2$, by sending $t^n(\log t)^m$ to $(n,m)$. In this notation,
\[
\trop_v(F)=(1,0)x'+(0,0)x+(1,0).
\]
Differentiating once, we obtain
\[
\trop_v(dF)=(1,0)x''+(0,0).
\]
Hence there is no tropical solution in $\B[\![t]\!]$, which agrees with the classical fact that $[F]$ has no formal power series solutions.

Now let $A\in \B[\![\mathcal M]\!]$ be a tropical solution of $\trop_v([F])$. From the equation $\trop_v(F)$ we see that the leading term of $A$ must be either $t\log t$ or $t+t\log t$. Moreover, for every $k\geq 0$ one has
\[
\trop_v(d^kF)=(1,0)x^{(k+1)}+(0,0)x^{(k)},
\]
and evaluating these tropical differential polynomials imposes no further restriction on higher terms. In particular, if $\varphi$ is a classical solution of $[F]$ and
\[
A=\trop_{\widetilde v}(\varphi)+B
\]
with $B\in\B[\![\mathcal M]\!]$ satisfying $\Phi(B)>t\log t$, then
\[
\Phi(d^kA)=\Phi(d^k\trop_{\widetilde v}(\varphi))
\qquad\text{for all }k\in\N.
\]
Since $\trop_{\widetilde v}(\varphi)$ is a tropical solution, the same is true for $A$. Thus there are strictly more tropical solutions than tropicalizations of classical solutions, showing that the fundamental theorem fails for transseries.
\end{ex}

\begin{oss}
The obstruction is not merely the non-reducedness of the pair $\Phi$. The essential point is that, over formal power series, repeated differentiation eventually sends every monomial of sufficiently high order to $0$, whereas terms such as $t^{-1}$ or $\log(t)$ do not exhibit this behaviour. Thus the same phenomenon persists even after passing to the reduction of the tropical pair.
\end{oss}

\subsubsection*{Transformations}
We next explain why the pullback transformations used for Puiseux series cannot, in general, reduce transseries involving logarithms to formal power series. Consider first the simple case of transseries of the form
\[
\varphi=\sum_{i\geq 0} c_i(\log t)t^i\in K[\log(t)][\![t]\!].
\]
One might hope to find a transformation on the independent variable that eliminates the logarithms. However, if we insist on the basic properties used in the Puiseux case, then this is impossible.

Indeed, a rational map cannot transform $\log(t)$ into a formal power series. Allowing more general substitutions does not help much: the natural choice $s(t)=\exp(t)$ sends $\log(t)$ to $t$, but it violates the condition $s(0)=0$. On the other hand, if one keeps $s(0)=0$, then any substitution of the form $s(t)=t\,s_0(t)$ with $s_0(t)\in K[\![t]\!]$ yields
\[
\log(s(t))=\log(t)+\log(s_0(t)),
\]
which is still not a formal power series.

A second possibility is to keep $s(t)=t$ and transform only the dependent variable. For instance, one might try a map such as
\[
z(t,x)=\exp(x/t^k)
\qquad (k\in\Z),
\]
which sends $x=t^k\log(t)$ to a monomial. But this does not work uniformly on all logarithmic transseries: for example, if $x=(t+1)\log(t)$, then
\[
z(t,x)=t^{t+1}
\qquad\text{or}\qquad
z(t,x)=t^{t-1},
\]
depending on the choice of $k$, and neither of these is a formal power series. Thus there is no evident analogue of the Puiseux pullback construction that would transform all transseries in $K[\log(t)][\![t]\!]$ into formal power series.

\section{Other applications}
The pullback transformations used above for Puiseux series can also be applied outside tropical differential algebra. One direct application arises in approximation theory and yields the following consequence of the Denef--Lipshitz theorem~\cite[Theorem 3.1]{denef1984power}.

\begin{teorema}\label{thm-DenefLipshitz}
Let $K \in \{ \C,\R,\Q_p \}$. Let $G$ be a finite differential system in $\Q[t]\{x_1,\ldots,x_n\}$, let $\sigma$ be a pullback transformation in the ordinary case for fixed integers $l \in \Z$ and $d \in \Z_{>0}$. Then there is an algorithm for deciding whether there exists a Puiseux series solution
\[
(\varphi_1,\ldots,\varphi_n)\in K(\!(t)\!)_d^n
\]
of $G$, with orders bounded below by $l$, whose transform under $\sigma$ is a formal power series solution of the transformed system $G^{\sigma}$.
\end{teorema}

\begin{proof}
Let $G^{\sigma}$ be the differential system obtained from $G$ by pullback along $\sigma$ and clearing denominators. By the ordinary-case specialization of the pullback construction developed above, together with Lemma~\ref{lem:pullback-solutions}, a tuple
\[
\varphi\in K(\!(t)\!)_d^n
\]
with orders bounded below by $l$ is a solution of $G$ if and only if its transform under $\sigma$ is a formal power series solution of $G^{\sigma}$. Therefore the existence of such Puiseux series solutions is equivalent to the existence of a formal power series solution of the transformed system. The latter can be decided algorithmically by~\cite[Theorem 3.1]{denef1984power}.
\end{proof}

\begin{oss}
In concrete situations, bounds for the ramification index $d$ and the order $l$ can often be read off from the steps of the Newton polygon method for differential equations; see e.g.~\cite{aroca2003power}. The theorem above is conditional on such prescribed bounds.
\end{oss}

\begin{ex}
Consider the differential equation
\[
F=xx'-1
\]
together with the initial condition $x(0)=0$.\footnote{In
Theorem~\ref{thm-DenefLipshitz} one can additionally impose finitely many
equations and inequations on the Taylor coefficients, as is clear from the
proof of~\cite[Theorem 3.1]{denef1984power}.}
Its solutions are
\[
\varphi=\pm\sqrt{2t}\in\Q[\![t^{1/2}]\!],
\]
which are not formal power series. They have order bounded below by $0$.

The pullback along
\[
\sigma_{2,0}(s,z)=(t,x)=(s^2,z)
\]
produces the transformed equation
\[
F^{\sigma_{2,0}}=zz'-2s.
\]
The formal power series
\[
\psi=\pm\sqrt{2}\,s
\]
is a solution of this transformed equation. Thus the Puiseux series
solutions of $F$ are detected via formal power series solutions of the
transformed equation.
\end{ex}

Other results concerning formal power series solutions can be transferred by the same method. 
Let $G$ be a finite differential system in the ordinary case, fix $d\in\mathbb Z_{>0}$ and $l\in\mathbb Z$, and let
\[
\sigma_{d,l}(s,z)=(t,x)=(s^d,s^{dl}z)
\]
be the corresponding pullback transformation. The formal Puiseux series solutions of $G$ with orders bounded below by $l$ are thereby put in bijection with the formal power series solutions of the transformed system $G^{\sigma_{d,l}}$. 
Consequently, results on jet spaces of $G^{\sigma_{d,l}}$ and their relation with formal power series solutions, see e.g.~\cite{moosa2010jet,seiler2010involution}, provide an approximative description of the corresponding formal Puiseux series solutions of $G$.

Likewise, following~\cite[Section~9]{seiler2010involution}, one may study sufficiently differentiable solutions of $G$ whose asymptotic behaviour is described by a Puiseux polynomial---the \emph{geometric solutions} in the sense of~\cite{seiler2010involution}---by first passing to the transformed formal power series setting.


\subsection*{Acknowledgments}
This project was initiated during our time together at the Max Planck Institute for Mathematics in the Sciences in Leipzig, Germany. We would like to thank our colleagues there for the stimulating and pleasant working environment. 
The first author was partially supported by the grant PID2020-113192GB-I00 (Mathematical Visualization: Foundations, Algorithms and Applications) from the Spanish MICINN and the OeAD project FR 09/2022.

\subsection*{Conflict of interest}
All authors declare no conflicts of interest in this paper.

\bibliographystyle{alpha}

\newcommand{\etalchar}[1]{$^{#1}$}

\end{document}